\documentclass[review]{article}

\usepackage[affil-it]{authblk}
\usepackage{amsmath, amssymb, amsfonts, amsthm, pgf, tikz, enumerate}
\usepackage{bm}

\theoremstyle{plain}
\newtheorem{theorem}{Theorem}[section]

\newtheorem{lemma}[theorem]{Lemma}
\newtheorem{corollary}[theorem]{Corollary}
\newtheorem{observation}[theorem]{Observation}

\theoremstyle{definition}

\newtheorem{example}[theorem]{Example}

\newtheorem{remark}[theorem]{Remark}

\renewcommand{\i}{{\rm i}}
\newcommand{\ST}{\, \big | \,}
\newcommand{\suchthat}{ \,\, | \,\, }
\DeclareMathOperator{\at}{\,{\rule[-3mm]{.1mm}{8mm}}\,}
\DeclareMathOperator{\jac}{Jac}

\title{Existence of a Not Necessarily Symmetric Matrix with Given Distinct Eigenvalues and Graph}
\author{Keivan Hassani Monfared}
\affil{University of Calgary\\ k1monfared@gmail.com}
\date{}

\begin{document}
\maketitle

\begin{abstract}
For given distinct numbers $\lambda_1 \pm \mu_1 \i, \lambda_2 \pm \mu_2 \i, \ldots, \lambda_k \pm \mu_k \i \in \mathbb{C} \setminus \mathbb{R}$ and $\gamma_1, \gamma_2, \ldots, \gamma_l \in \mathbb{R}$, and a given graph $G$ with a matching of size at least $k$, we will show that there is a real matrix whose eigenvalues are the given numbers and its graph is $G$. In particular, this implies that any real matrix with distinct eigenvalues is similar to a real, irreducible, tridiagonal matrix.\\

{Keywords: Inverse Eigenvalue Problem, Graph, Jacobian Method, Implicit Function Theorem, Transversal Intersection\\

}
\end{abstract}

\section{Introduction}
	
	A \emph{directed graph} $G = (V,E)$ is a pair of sets $V$ and $E$ where $V$ is the set of vertices of $G$, and $E$, the set of edges of $G$, is a subset of $V \times V$. That is, each element of $E$ is an ordered pair $(u,v)$, with $u,v \in V$. We say a graph $G$ is \emph{loopless} when for each $(u,v) \in E$, we have $u \neq v$. In this paper we only consider loopless graphs. If $(u,v) \in E$ then we say $u$ is adjacent to $v$ and denote it by $u \rightarrow v$. Note that such a graph might have both edges $(u,v)$ and $(v,u)$, but since $E$ is a set, there are no multiple edges from $u$ to $v$. A directed loopless graph $G = (V,E)$ is said to have a \emph{matching} of size $k$ if $E$ contains $k$ vertex-disjoint edges $(u_1,v_1), \ldots, (u_k,v_k)$ and their reverses $(v_1,u_1), \ldots, (v_k,u_k)$.
	
	If for each $u \neq v$ the edge $(u,v) \in E$ if and only if $(v,u) \in E$, then $G$ is bidirected. Hence we can ignore the directions of edges and consider $E$ as a set of 2-subsets of $V$. That is, $E \subset \big\{ \{ u,v \} \ST u,v \in V \big\}$. In this case we call $G$ an \emph{undirected graph}. An undirected loopless graph $G=(v,E)$ is said to have a \emph{matching} of size $k$ if $E$ contains $k$ vertex-disjoint edges $\{u_1,v_1\}, \ldots, \{u_k,v_k\}$
	
	Let $A \in \mathbb{R}^{n\times n}$. We say a (directed or undirected) loopless graph $G$ is the \emph{graph of the matrix} $A$ when for each $i\neq j$ we have $A_{i,j} \neq 0$ if and only if $i \rightarrow j$. Note that the diagonal entries of $A$ can be zero or nonzero.

	It is of interest to study the existence of matrices with given spectral properties and graph, see \cite[Chapter 4]{chugolub05}.	 For the problems when the solution matrix is not necessarily symmetric see \cite{chugolub02} for a survey on the structured inverse eigenvalue problems with an extensive bibliography, specially SIEP6b, and see \cite{minranknonsymm09} for the minimum rank problems. In Section \ref{SectionPreliminaries} we provide some machinery in order to prove the main theorem in Section \ref{SectionMainResult} about the existence of a solution for the inverse eigenvalue problem for a graph when the solution matrix is not necessarily symmetric.

\section{Preliminaries}\label{SectionPreliminaries}
In this section we first introduce the notion of transversality, and use it to show simple real roots of a real polynomial remain real under small perturbations. Then we give an example of a real matrix whose spectrum is a given set of real numbers and pairs of complex conjugate numbers. For the given matrix we define a neighborhood of its spectrum and put an order on it. We finally study the how small perturbations of the matrix changes its simple eigenvalues.

	\subsection{Transversal intersections}
	
	Here we first define two families of manifolds and show that they intersect transversally at some points. Then we will use this result to show that small perturbation of a real polynomial does not change the number of its simple real roots.

	Let 
	\begin{equation}\label{DefPeeTeeEx}
	p_t(x) = x^n + a_{n-1}(t) x^{n-1} + \cdots + a_1(t) x + a_0(t) \in \mathbb{R}[x],
	\end{equation}
	where for each $i = 0, 1, \ldots, n-1$ the coefficient $a_i(t)$ is a continuous function of $t$ from $(-1,1)$ to $\mathbb{R}$. For each $t \in (-1,1)$ define 
	\begin{equation}\label{DefPeeTee}
	P(t) = \{ (x, p_t(x)) \in \mathbb{R}^2 \ST x \in \mathbb{R}\},
	\end{equation}
	and 
	\begin{equation}\label{DefEss}
	S = \{ (x,0) \in \mathbb{R}^2 \ST x \in \mathbb{R} \}.
	\end{equation}

	Note that $S$ and $P(t)$ for each $t \in (-1,1)$ are smooth manifolds of $\mathbb{R}^2$. The tangent space to $S$ at any point $ (x_0,0) \in \mathbb{R}^2$, $\mathcal{T}_{S.(x_0,0)}$, is $S$ itself, and the tangent space to $P(t)$ at any point $(x_0,p_t(x_0))$, $\mathcal{T}_{P(t).(x_0,p_t(x_0))}$ is the tangent line to the graph of $y = p_t(x)$ at the point $(x_0,p_t(x_0))$. The latter tangent space is the set 
	\[ 
	\mathcal{T}_{P(t).(x_0,p_t(x_0))} = \{ (x_0, p'_t(x_0)(x-x_0)+p_t(x_0) \ST x \in \mathbb{R}\},
	\]
	where $p'_t(x_0)$ denotes the derivative of $p_t(x)$ evaluated at $x_0$. It is evident that $\mathcal{T}_{P(t).(x_0,p_t(x_0))}$ is a line not parallel to $S$ when $x_0$ is not a root of $p'_t(x)$. In particular, when $x_0$ is a root of $p_t(x)$, then $P(t)$ and $S$ intersect transversally at $x_0$ if and only if $x_0$ is a simple root of $p_t(x)$. We shall need the following special case of \cite[Lemma 2.1]{schrijver99}.
	
\begin{lemma}\label{manifoldversionofIFT}
	Let $P(t)$ and $S(t)$ be smooth families of manifolds in $\mathbb{R}^N$, for some positive integer $N$, and assume that $P(0)$ and $S(0)$ intersect transversally at $x$. Then there exists a neighborhood $W \subseteq \mathbb{R}^2$ of the origin, such that for each $\bm \varepsilon = (\varepsilon_1, \varepsilon_2) \in W$, the manifolds $P(\varepsilon_1)$ and $S(\varepsilon_2)$ intersect transversally at a point $x(\varepsilon)$, so that $x(0)=x$ and $x(\varepsilon)$ depends continuously on $\varepsilon$.
\end{lemma}

	The following lemma shows that if $p(x)$ is a polynomial in $\mathbb{R}[x]$ with $k$ simple real roots, then any sufficiently small perturbation of $p(x)$ also has $k$ simple real roots.

\begin{lemma}\label{SmallPerturbationOfPollynomial}
	Let $p_t(x)$ be defines as in Equation \eqref{DefPeeTeeEx}. If $p_0(x)$ has a simple real root, then there is $\varepsilon > 0$ such that for each $-\varepsilon < t < \varepsilon$ the polynomial $p_t(x)$ has a simple real root.
\end{lemma}
\begin{proof}
	Let $x_0$ be a simple root of $p_0(x)$, and let $S$ and $P(t)$ be defines be Equations \eqref{DefPeeTee} and \eqref{DefEss}. Then $S$ and $P(0)$ intersect transversally at $x_0$, and thus by Lemma \ref{manifoldversionofIFT} there is an $\varepsilon > 0$ such that for any $-\varepsilon < t < \varepsilon$ the manifolds $P(t)$ and $S$ intersect transversally at $x_0(t)$ where $x_0(0) = x_0$ and $x_0(t)$ depends continuously on $t$. In particular, $x_0(t)$ is a simple root of $p_t(x)$.
\end{proof}

\begin{corollary}\label{CorNumOfRootsFixed}
	If $p(x)$ is a polynomial of degree $n$ in $\mathbb{R}[x]$ with $k$ simple real roots and $n-k$ distinct non-real roots, then any sufficiently small perturbation of $p(x)$ also has $k$ simple real roots and $n-k$ distinct non-real roots.
\end{corollary}

	\subsection{A matrix with a given spectrum}
	
	Now, we show that for given $l$ distinct real numbers, $2k$ distinct non-real numbers which are conjugate pairs, there is a real matrix whose eigenvalues are the given numbers. Then we define some (ordered) neighbourhood of its spectrum.
	
\begin{example}\label{TheMatrixAy}
	There is a real matrix $A$ whose eigenvalues are given numbers $\lambda_j \pm \mu_j \i \in \mathbb{C} \setminus \mathbb{R}$ for $j = 1,2,\ldots, k$, and $\gamma_j \in \mathbb{R}$ for $j = 1,2,\ldots, l$.
	Let 
	\[
	A = \left( \bigoplus_{j=1}^{k} \left[ \begin{array}{cc}
	\lambda_j & \mu_j\\
	-\mu_j & \lambda_j
	\end{array} \right] \right)
	\oplus
	\left( \bigoplus_{j=1}^{l}
	\left[ \begin{array}{c}
	\gamma_j
	\end{array} \right] \right).
	\]
	Note that a unit eigenvector corresponding to the eigenvalue $\lambda_j \pm \mu_j \i$ is
	\begin{center}
	\begin{tikzpicture}
		\node[] () at (0,0) {$
	\bm v_j = \frac{1}{\sqrt{2}}\left[ \begin{array}{c}
	0\\
	\vdots\\
	0\\
	1\\
	\pm \i\\
	0\\
	\vdots\\
	0	
	\end{array} \right].
	$};
		\draw[->] (2.2,-.2) -- (.8,-.2);
		\draw[->] (1.9,.3) -- (.8,.3);
		\node[scale=.7] () at (2.5,-.2) {$2j$};
		\node[scale=.7] () at (2.5,.3) {$2j-1$};
	\end{tikzpicture}
	\end{center}
	Furthermore, note that the corresponding eigenvector of the same eigenvalue for $A^\top$ is $\bm w_j = \overline{\bm v}_j$.
\end{example}

\begin{remark}
	Note that this simple example shows that any real matrix is similar to a tridiagonal real matrix.
\end{remark}

Now, we define a matrix of variables for a graph $G$ and we will consider the rate of change of its eigenvalues as the variables change. Let $G$ be a graph on $n = 2k+l$ vertices and $k+m$ edges. Assume that $G$ has a matching $\mathcal{M} = \big\{ \{1,2\}, \{3,4\}, \ldots, \{2k-1,2k\} \big\}$, and let the rest of the edges of $G$ be denoted by $\{i_r,j_r\}$, with $i_r < j_r$ for $r = 1,2,\ldots,m$. Also, let $\bm x = (x_1,x_2,\ldots,x_k), \bm y = (y_1,y_2,\ldots,y_k) \in \mathbb{R}^k$, $\bm z = (z_1,z_2,\ldots,z_l) \in \mathbb{R}^l$, and $\bm u = (u_1,u_2,\ldots,u_m), \bm \omega = (\omega_1,\omega_2,\ldots,\omega_m) \in \mathbb{R}^m$. 

Define $M = M(\bm x, \bm y, \bm z, \bm u, \bm \omega)$ with $x_j$ on the $(2j-1,2j-1)$ and $(2j,2j)$ positions and $y_j$ on the $(2j-1,2j)$ position and $-y_j$ on the $(2j,2j-1)$ position, for $j = 1,2,\ldots,k$; and $z_j$ on the $(2k+j,2k+j)$ position, for $j=1,2,\ldots,l$; and $u_r$ on $(i_r,j_r)$ position, and $\omega_r$ on $(j_r,i_r)$ position, for $r= 1,2,\ldots,m$. The matrix $M$ has the following form
	\[
	M = \left[ \begin{array}{c|c|c|c}
	\begin{array}{cc}
		x_1 & y_1\\
		-y_1 & x_1	
	\end{array} &&&\\ \hline
	& \begin{array}{c}
	\ddots
	\end{array} &&\\ \hline
	&& \begin{array}{cc}
		x_k & y_k\\
		-y_k & x_k
	\end{array} & \\ \hline
	&&& \begin{array}{ccc}
		z_1 &&\\
		& \ddots\\
		&& z_l
	\end{array}
	\end{array} \right],
	\]
where the entries not shown are either $0$, some $u_r$, or some $\omega_r$. Note that 
	\[M(\lambda_1,\ldots, \lambda_k, \mu_1,\ldots,\mu_k, \gamma_1,\ldots,\gamma_l, 0,\ldots,0) = A,\] where $A$ is the matrix in Example \ref{TheMatrixAy}.
	
Now, and for the rest of this paper, assume that 
	\[\Lambda = \{ \lambda_j \pm \mu_j \i \in  \mathbb{C} \setminus \mathbb{R} \suchthat j=1,2,\ldots, k \} \cup \{ \gamma_j \in \mathbb{R} \suchthat j = 1,2,\ldots,l \}
	\]
is a fixed set of $n = 2k+l$ distinct numbers. Define an $\varepsilon$-neighborhood of a set $S \subseteq \mathbb{C}$ to be the set of points that are of distance at most $\varepsilon$ from a point of $S$, that is,
	\[
	N_{\varepsilon}(S) = \{ z \in \mathbb{C} \suchthat |z - s| < \varepsilon \text{ for some } s \in S  \}.
	\]
Since $\Lambda$ consists of $n$ distinct points in the complex plain, there is an $\varepsilon$ such that $N_\varepsilon(\Lambda)$ consists of $n$ disjoint discs $\mathbb{D}_1^+, \mathbb{D}_2^+, \ldots, \mathbb{D}_k^+$, $\mathbb{D}_1^-, \mathbb{D}_2^-, \ldots, \mathbb{D}_k^-$, and $\mathbb{D}_1, \mathbb{D}_2, \ldots, \mathbb{D}_l$ where $\mathbb{D}_j^+$ contains $\lambda_j + \mu_j \i$, $\mathbb{D}_j^-$ contains $\lambda_j - \mu_j \i$, and $\mathbb{D}_j$ contains $\gamma_j$. Also, let $\mathbb{D}_j'$ denote $\mathbb{D}_j \cap \mathbb{R}$, and 
	\[
	\mathbb{D} = \left( \bigcup_{j=1}^{k} \mathbb{D}_j^+ \right) \cup \left( \bigcup_{j=1}^{k} \mathbb{D}_j^- \right) \cup \left( \bigcup_{j=1}^{k} \mathbb{D}_j' \right).
	\]
	
\begin{observation}\label{nonorthogonalevectors}
	Let $A \in \mathbb{R}^{n\times n}$ have $n$ distinct eigenvalues. Then for an eigenvalue $\lambda$ and corresponding left eigenvector $\bm u^\top$ and right eigenvector $\bm v$, we have $\bm u^\top \bm v \neq 0$.
\end{observation}
\begin{proof}
	Let $J$ be the Jordan canonical form of $A$. Since all the eigenvalues of $A$ are simple, $J$ is diagonal. Let $A = SJS^{-1}$ for some invertible matrix $S$. Observe that for an eigenvalue $\lambda$ there is an $1 \leq i \leq n$ such that the $i$-th column of $S$, $\bm s^i$, is a right eigenvector of $A$ for the eigenvalue $\lambda$, and the $i$-th row of $S^{-1}$, ${\bm s_i}^\top$, is a left eigenvector of $A$ for the eigenvalue $\lambda$. Since $S^{-1} S = I$ we have ${\bm s_i}^\top \bm s^i = 1$. This implies $\bm u^\top \bm v \neq 0$. 
\end{proof}

	\subsection{Small perturbations of a matrix and its eigenvalues}

	In this part we study the effect of small perturbations of a matrix on its eigenvalues and define a function that maps a matrix to its eigenvalues. Then we will show that the Jacobian matrix of this function evaluated at a certain point has full rank.

	If the matrix $M$ is in a small neighborhood of $A$, then its eigenvalues lie in $N_\varepsilon(\Lambda)$. Moreover, Lemma \ref{CorNumOfRootsFixed} implies that the real eigenvalues of $M$ lie in $\mathbb{D}_j'$'s. Let $\lambda_j(M)$ denote the real part of the eigenvalue of $M$ that lies in $\mathbb{D}_j^+$, $\mu_j(M)$ denote the imaginary part of the eigenvalue of $M$ that lies in $\mathbb{D}_j^+$, and $\gamma_j(M)$ denote the real eigenvalue of $M$ that lies in $\mathbb{D}_j'$. 

	Now define a function $f$ in a small neighborhood of $A$ as follows:
	\begin{align}\label{FunctionEff}
	f \colon \mathbb{R}^{(2k+l)+2m} &\to \mathbb{R}^{2k+l}\\
	M &\mapsto \big( \lambda_1(M),\ldots,\lambda_k(M) \, , \, \mu_1(M), \ldots, \mu_k(m) \, , \, \gamma_1(M), \ldots, \gamma_l(M) \big).
	\end{align}
Thus, $f$ maps a small neighborhood of $A$ to $\mathbb{D}$. The goal is to show that the Jacobian of this function has full row rank. The following two lemmas calculates the derivative of each of components of $f$.

\begin{lemma}\label{LemmaMainDerivatives}
	Let $A$ and $B$ be real matrices where $A$ has distinct eigenvalues $\lambda_r \pm \mu_r \i \in \mathbb{C} \setminus \mathbb{R}$, for $r = 1,2,\ldots,k$, and $\gamma_r \in \mathbb{R}$, for $r = 1,2,\ldots,l$. let $\bm v_r$'s be corresponding unit eigenvectors of $A$, and $\bm w_r$'s be corresponding unit eigenvectors of $A^\top$. Also, Let $A(t) = A + t B$, for $t \in (-1,1)$. Then 
	the followings hold:
	\[
	\frac{\partial}{\partial t}\lambda_j(A(t)) = {\rm Re}(\zeta)\quad \text{ and } \quad
	\frac{\partial}{\partial t}\mu_j(A(t)) = {\rm Im}(\zeta),
	\]
	where $\zeta = \displaystyle \frac{\bm w_r^\top B \bm v_r}{\bm w_r^\top \bm v_r}$.
\end{lemma}
\begin{proof}
	Let $\bm v_r(t)$ be a unit eigenvector of $A(t)$ corresponding to $\lambda_r(t) + \mu_r(t) \i$. Note that 
	\[
	A(t) \to A, \quad \bm v_r(t) \to \bm v_r, \quad \lambda_r(t) \to \lambda_r, \quad \text{ and }	\mu_r(t) \to \mu_r,
	\]
	as $t \to 0$.
	Note that 
	\[
	A(t) \bm v_r(t) = (\lambda_r(t) + \mu_r(t) \i) \bm v_r(t).
	\]
	Differentiating with respect to $t$ we have 
	\[
	\dot{A}(t) \bm v_r(t) + A(t) \dot{\bm v}_r(t) = (\dot{\lambda}_r(t) + \dot{\mu}_r(t) \i) \bm v_r(t) + (\lambda_r(t) + \mu_r(t) \i) \dot{\bm v}_r(t).
	\]
	Letting $t = 0$ we have
	\[
	B \bm v_r + A \dot{\bm v}_r(0) = (\dot{\lambda}_r(0) + \dot{\mu}_r(0) \i) \bm v_r + (\lambda_r + \mu_r \i) \dot{\bm v}_r(0).
	\]
	Multiply both sides by $\bm w_r^\top$ from left
	\[
	\bm w_r^\top B \bm v_r + \bm w_r^\top A \dot{\bm v}_r(0) = (\dot{\lambda}_r(0) + \dot{\mu}_r(0) \i) \bm w_r^\top \bm v_r + (\lambda_r + \mu_r \i) \bm w_r^\top \dot{\bm v}_r(0).
	\]
	since $ \bm w_r^\top A = (\lambda_r + \mu_r \i) \bm w_r^\top$ we get
	\[
	\bm w_r^\top B \bm v_r +  (\lambda_r + \mu_r \i) \bm w_r^\top \dot{\bm v}_r(0) = (\dot{\lambda}_r(0) + \dot{\mu}_r(0) \i) \bm w_r^\top \bm v_r + (\lambda_r + \mu_r \i) \bm w_r^\top \dot{\bm v}_r(0).
	\]
	The second terms in left hand side and right hand side of the equation are equal. Thus
	\[
	\bm w_r^\top B \bm v_r = (\dot{\lambda}_r(0) + \dot{\mu}_r(0) \i) \bm w_r^\top \bm v_r.
	\]
	By Observation \ref{nonorthogonalevectors} we have $\bm w_r^\top \bm v_r \neq 0$. Hence
	\begin{equation}\label{DerivativeEquatioinOne}
	\dot{\lambda}_r(0) + \dot{\mu}_r(0) \i = \frac{\bm w_r^\top B \bm v_r}{\bm w_r^\top \bm v_r}.
	\end{equation}
	Conjugating both sides we get 
	\begin{equation}\label{DerivativeEquatioinTwo}
	\dot{\lambda}_r(0) - \dot{\mu}_r(0) \i = \frac{\overline{\bm w}_r^\top B \overline{\bm v}_r}{\overline{\bm w}_r^\top \overline{\bm v}_r}.
	\end{equation}
	Now once add equations \eqref{DerivativeEquatioinOne} and \eqref{DerivativeEquatioinTwo} and once subtract them to get 
	\begin{equation}\label{DerivativesOne}
	\dot{\lambda}_r(0) = {\rm Re} \left( \frac{\bm w_r^\top B \bm v_r}{\bm w_r^\top \bm v_r}\right),
	\end{equation}
	\begin{equation}\label{DerivativesOne}
	\dot{\mu}_r(0) = {\rm Im} \left( \frac{\bm w_r^\top B \bm v_r}{\bm w_r^\top \bm v_r}\right).
	\end{equation}
\end{proof}

\begin{lemma}
Let $A$ be the matrix in Example $\ref{TheMatrixAy}$, and let $E_{ij}$ denote the matrix of appropriate size with a $1$ on its $(i,j)$-entry and zeros elsewhere.  Also, let $B$ in Lemma $\ref{LemmaMainDerivatives}$ be one of the followings:
	\begin{enumerate}
		\item\label{firstcase} $B = E_{2j-1,2j-1} + E_{2j,2j}$, for some $j = 1,2,\ldots,k$,
		\item\label{secondcase} $B = E_{2j-1,2j} - E_{2j,2j-1}$, for some $j = 1,2,\ldots,k$, or
		\item\label{thirdcase} $B = E_{2k+j,2k+j}$, for some $j = 1,2,\ldots,l$.
	\end{enumerate}
	Then 
	\[
	\frac{\partial}{\partial t}\lambda_r(A(0)) = 
	\begin{cases}
		1; & \text{if $t$ is in $(2r-1,2r-1)$ or $(2r,2r)$ position} \\
		  & \text{for $r = 1,2,\ldots k$ or $r$ is in $(2k+r,2k+r)$ }\\
		  & \text{position for $r = 1,2,\ldots,l$.}\\
		0; & \text{otherwise,}
	\end{cases}
	\]
	and 
	\[
	\frac{\partial}{\partial t}\mu_r(A(0)) = 
	\begin{cases}
		1; & \text{if $t$ is in $(2r-1,2r)$ or $(2r,2r-1)$ position} \\
		  & \text{for $r = 1,2,\ldots k$,} \\
		0; & \text{otherwise.}
	\end{cases}
	\]
\end{lemma}
\begin{proof}
	Note that for matrix $A$ in Example \ref{TheMatrixAy} $w_r = \overline{v}_r$. Thus $w_r^\top v_r = 1$ and 
	\[
	w_r^\top B v_r = 
	\begin{cases}
		\bm w_{r_{2j-1}} \bm v_{r_{2j-1}} + \bm w_{r_{2j}} \bm v_{r_{2j}}  & \text{(in case \ref{firstcase}),} \\
		\bm w_{r_{2j-1}} \bm v_{r_{2j}} - \bm w_{r_{2j}} \bm v_{r_{2j-1}}  & \text{(in case \ref{secondcase}),} \\
		\bm w_{r_{2k+j}} \bm v_{r_{2k+j}}						 & \text{(in case \ref{thirdcase}).} 
	\end{cases}
	\]
	Thus 
	\[
	w_r^\top B v_r = 
	\begin{cases}
		1; \text{ if and only if } r = j & \text{(in case \ref{firstcase}),} \\
		\i; \text{ if and only if } r = j& \text{(in case \ref{secondcase}),} \\
		1; \text{ if and only if } r = 2k+j & \text{(in case \ref{thirdcase}).} 
	\end{cases}
	\]
\end{proof}

	Now we are ready to evaluate the Jacobian of the function $f$.

\begin{corollary}\label{JacobianIsI}
	Let $A$ be the matrix in Example $\ref{TheMatrixAy}$, and the function $f$ be defined by Equation \eqref{FunctionEff}. Also, let $\jac_{\bm x, \bm y, \bm z}$ denote the matrix obtained from the Jacobian matrix of $f$ by keeping only the columns corresponding to the derivatives with respect to $x_j$'s, $y_j$'s, and $z_j$'s. Then 
	\[ \jac_{\bm x,\bm y,\bm z}(f) \at_{A} = I, \]
	where $I$ denotes the identity matrix of size $2k+l$, and thus it is nonsingular.
\end{corollary}

\section{Main Result}\label{SectionMainResult}
In this section we prove that for given $l$ distinct real numbers, $2k$ distinct non-real numbers which are conjugate pairs, and a graph on $n$ vertices with a matching of size at least $k$, there is a real matrix whose eigenvalues are the given numbers and its graph is the given graph.

\begin{theorem}
	For given distinct numbers $\lambda_1 \pm \mu_1\i, \lambda_2 \pm \mu_2\i, \ldots, \lambda_k \pm \mu_k\i \in \mathbb{C} \setminus \mathbb{R}$ and $\gamma_1, \gamma_2, \ldots, \gamma_l \in \mathbb{R}$, and a given graph $G$ on $2k+l$ vertices with a matching of size at least $k$ there is a real matrix whose eigenvalues are the given numbers and its graph is $G$.
\end{theorem}
\begin{proof}
	Let $A$ be the matrix in Example \ref{TheMatrixAy}, matrix $M$ be defined as above, and function $f$ be defined by Equation \eqref{FunctionEff}. Let $\bm \lambda = (\lambda_1,\ldots, \lambda_k)$, $\bm \mu = (\mu_1,\ldots,\mu_k)$, and $\bm \gamma = (\gamma_1,\ldots,\gamma_l)$. also, let $\bm 0$ denote a zero vector of appropriate size. Note that 
	\[
	M(\bm \lambda, \bm \mu, \bm \gamma, \bm 0, \bm 0) = A.
	\]
	Also note that
	\[
	f(\bm \lambda, \bm \mu, \bm \gamma, \bm 0, \bm 0) = (\bm \lambda, \bm \mu, \bm \gamma).
	\]
	Furthermore, by Corollary \ref{JacobianIsI} we have \[ \jac_{\bm x,\bm y,\bm z}(f) \at_{A} = I,\] and hence it is nonsingular. Then by the Implicit Function Theorem for any small $\bm \varepsilon, \bm \delta \in \mathbb{R}^m$, there are $\overline{\bm \lambda}, \overline{\bm \mu}, \overline{\bm \gamma}$ close to $\bm \lambda, \bm \mu, \bm \gamma$ such that $f(\overline{\bm \lambda}, \overline{\bm \mu}, \overline{\bm \gamma}, \bm \varepsilon, \bm \delta ) = (\bm \lambda, \bm \mu, \bm \gamma)$. Choose $\overline{\bm \lambda}$, $\overline{\bm \mu}$, and $\overline{\bm \gamma}$ such that they have no zero entries, and let $\tilde{A} = M(\overline{\bm \lambda}, \overline{\bm \mu}, \overline{\bm \gamma}, \bm \varepsilon, \bm \delta)$. Then eigenvalues of $\tilde{A}$ are $\lambda_j \pm \mu_j \i$ for $j = 1,2,\ldots,k$ and $\gamma_j$ for $j = 1,2,\ldots l$, and graph of $\tilde{A}$ is $G$.
\end{proof}

\begin{remark}
	If all the prescribed eigenvalues are real, one can always choose $\bm \varepsilon = \bm \delta$ to find a symmetric matrix $\tilde{A}$. Also, if all the prescribed eigenvalues are purely imaginary, one can always choose $\bm \varepsilon = - \bm \delta$ to make $\tilde{A}$ the sum of a skew-symmetric matrix and a diagonal matrix. The case with all real eigenvalues was previously proven in \cite{lambdamu13} and the case with all purely imaginary eigenvalues was shown in \cite{sudiptakeivan14}, and the constructed matrix is shown to have a zero diagonal, that is, it is a skew-symmetric matrix.
\end{remark}

\begin{corollary}
	For a given graph $G$ with a matching of size $k$, any real matrix with distinct eigenvalues which at most $2k$ of them are non-real, is similar to a real matrix whose graph is $G$.
\end{corollary}

	Note that this implies any real matrix with distinct eigenvalues is similar to a tridiagonal real matrix with nonzero superdiagonal and subdiagonal entries. On the other hand any real tridiagonal matrix with nonzero superdiagonal and subdiagonal entries has distinct eigenvalues. Thus we have the following corollary. 
\begin{corollary}
	A real matrix has distinct eigenvalues if and only if it is similar to a real irreducible tridiagonal matrix.
\end{corollary}

\bibliographystyle{plain}
\bibliography{ref151217}
\end{document}